\documentclass[11pt]{article}
\usepackage[dvips]{graphicx}
\usepackage{color}
\usepackage{amsmath}
\usepackage{amsfonts}
\usepackage{amssymb}
\usepackage{amsthm}
\usepackage{newlfont}
\usepackage{epsfig}
\usepackage{multirow}
\usepackage{setspace}
\usepackage{hyperref}
\usepackage{cite}

\begin{document}

\newtheorem{assumption}{Assumption}[section]
\newtheorem{definition}{Definition}[section]
\newtheorem{lemma}{Lemma}[section]
\newtheorem{proposition}{Proposition}[section]
\newtheorem{theorem}{Theorem}[section]
\newtheorem{corollary}{Corollary}[section]
\newtheorem{remark}{Remark}[section]

\title{A Stratum Approach to Global Stability of Complex Balanced Systems}
\author{David Siegel\thanks{Supported by a Natural Sciences and Engineering Research Council of Canada Research Grant} and Matthew D. Johnston \thanks{Supported by a Natural Sciences and Engineering Research Council of Canada Post-Graduate Scholarship \newline \textbf{Keywords:} chemical kinetics; stability theory; persistence; complex balancing; locking sets \newline \textbf{AMS Subject Classifications:} 80A30, 34D20, 37C75.}
  \vspace*{.2in} \\
Department of Applied Mathematics, University of Waterloo, \\
Waterloo, Ontario, Canada N2L 3G1 }
\date{}
\maketitle

\bigskip

\begin{abstract}

It has long been known that complex balanced mass-action systems exhibit a restrictive form of behaviour known as locally stable dynamics. This means that within each compatibility class $\mathcal{C}_{\mathbf{x}_0}$---the forward invariant space where solutions lies---there is exactly one equilibrium concentration and that this concentration is locally asymptotically stable. It has also been conjectured that this stability extends globally to $\mathcal{C}_{\mathbf{x}_0}$. That is to say, all solutions originating in $\mathcal{C}_{\mathbf{x}_0}$ approach the unique positive equilibrium concentration rather than $\partial \mathcal{C}_{\mathbf{x}_0}$ or $\infty$. To date, however, no general proof of this conjecture has been found.

In this paper, we approach the problem of global stability for complex balanced systems through the methodology of dividing the positive orthant into regions called strata. This methodology has been previously applied to detailed balanced systems---a proper subset of complex balanced systems---to show that, within a stratum, trajectories are repelled from any face of $\mathbb{R}_{\geq 0}^m$ adjacent to the stratum. Several known global stability results for detailed balanced systems are generalized to complex balanced systems.

\bigskip

\end{abstract}

\section{Introduction}

Chemical reaction modeling is rapidly becoming a topic of great interest in areas such as systems biology, atmospherics, pharmaceutics, industrial chemistry, etc., where mathematical tools are used to simplify, analyze, and illuminate behaviour of a variety of chemical-reaction-based natural phenomena. Consequently, many foundational concepts of such mathematical models have once again achieved prominence in the mathematical literature \cite{A,A3,C-D-S-S,S,S-C,S-M}.

One such foundational concept is that of complex balancing of chemical reaction networks, which has been used in analysis of industrial chemical reaction networks \cite{F2}. In 1972, the authors F. Horn and R. Jackson showed that relative to every compatibility class---the invariant space where solutions lie---complex balanced chemical reaction networks necessarily have exactly one positive equilibrium state and that this equilibrium state is asymptotically stable \cite{H-J1}. In conjunction with M. Feinberg, they also derived necessary and sufficient conditions for a system to be complex balanced based solely on the reaction graph of the system \cite{F1,H}. This work culminated in the \emph{Deficiency Zero Theorem} and was a substantial generalization of existing results on stability which required conservation of mass and balancing of forward and backwards reaction rates at equilibrium for each reaction.

It was theorized at the time that the convergence of solutions to the positive equilibrium state extended globally to the entire positive compatibility class, effectively eliminating the possibility that solutions converged to the boundary of the positive orthant. (Indeed, the point seemed so inextricably connected with asymptotic stability that in the original paper the authors errantly asserted that they had in fact proved just that! \cite{H-J1}) To date, however, the conjecture is only known to hold for certain special cases, which will not be summarized here. Important work has also been done in restricting the nature of any possible $\omega$-limit points on the boundary. In \cite{S-M}, the authors show that any $\omega$-limit point lying on the boundary is a complex balanced equilibrium concentration. In \cite{A} and \cite{A3}, the authors show that $\omega$-limit points may only lie on certain subsets of the boundary where these subsets can be easily determined by the reaction graph of the mechanism.

In this paper, we extend the stability results obtained in \cite{C-D-S-S}. In that paper, the authors showed that for detailed balanced mechanisms with bounded, two-dimensional compatibility classes, solutions originating in the positive orthant necessarily tend to the associated positive equilibrium concentration and not to the boundary. Their approach consisted of dividing the positive orthant into regions, called strata, and then manipulating the governing differential equations of the mechanism to show that within each stratum trajectories were repelled from the boundary. By generalizing their concept of strata, we will show how their results can be extended to complex balanced mechanisms.

The paper is organized as follows: in Section 2, we briefly introduce the relevant mathematical model for chemical reaction networks and present the notion of a complex balanced system; in Section 3, we extend the notion of strata and the linear Lyapunov functions $H(\mathbf{x}) = \langle \alpha, \mathbf{x} \rangle$ introduced in \cite{C-D-S-S} to complex balanced mechanisms, and present a few applications and examples; in Section 4, we give some concluding remarks including why we think this approach is a significant step towards proving the general \emph{Global Attractor Conjecture} (Proposition \ref{globalattractorconjecture}).

Throughout the paper, we will let $\mathbb{R}_{>0}^m$ and $\mathbb{R}_{\geq 0}^m$ denote the $m$-dimensional spaces with all coordinates strictly positive and non-negative, respectively.

\section{Background}

In this section, we outline the important concepts of chemical kinetics which will be needed throughout this paper. We introduce the concept of complex balancing first introduced in \cite{F1,H,H-J1} and outline the relevant results of these papers.

\subsection{Chemical Reaction Mechanisms}

Within the mathematical literature, several distinct ways to represent chemical reaction networks have been proposed. In this paper, we will follow closely the \emph{complex}-oriented formulation introduced by Horn \emph{et al.} in \cite{H-J1}. (For examples of \emph{reaction}- and \emph{species}-oriented formulations, see \cite{B-B1} and \cite{V-H}, respectively.)

An elementary chemical reaction consists of a set of reactants combining at some fixed rate to form some set of products. We will let $\mathcal{A}_j$ denote the \emph{species} or \emph{reactants} of the system and define $| \mathcal{S} | = m$ where $\mathcal{S}$ is the set of distinct species of the system. The set of all reactants or all products of a particular reaction are called \emph{complexes} and will be denoted $\mathcal{C}_i$. They are linear combinations of the species and therefore can be denoted $\mathcal{C}_i = \sum_{j=1}^m z_{ij} \mathcal{A}_j$ where $\mathbf{z}_i = [ z_{i1}, z_{i2}, \ldots, z_{im}] \in \mathbb{Z}_{\geq 0}^m$. We define $| \mathcal{C} | = n$ where $\mathcal{C}$ is the set of distinct complexes in the system.

It is convenient to represent the elementary reactions of our system not as a list of individual reactions, but as interactions between the $n$ distinct complexes of the system. In this setting, the reaction network can be represented as
\begin{equation}
\label{reaction2}
\mathcal{C}_i \; \stackrel{k(i,j)}{\longrightarrow} \; \mathcal{C}_j, \hspace{0.2in} \mbox{ for }i,j=1, \ldots, n,
\end{equation}
\noindent where $\mathcal{C}_i$ is the \emph{reactant complex}, $\mathcal{C}_j$ is the \emph{product complex}, and $k(i,j) \geq 0$ is the reaction rate associated with the reaction from $\mathcal{C}_i$ to $\mathcal{C}_j$ \cite{H-J1}. This representation of a chemical kinetics mechanism will be called the \emph{reaction graph}.

Note that if either $i=j$ or the mechanism does not contain a reaction with $\mathcal{C}_i$ as the reactant and $\mathcal{C}_j$ as the product, then $k(i,j)=0$. Otherwise, $k(i,j)>0$. The set of index pairs $(i,j)$ for which $k(i,j)>0$ will be denoted by $\mathcal{R}$ and the number of such index pairs will be denoted by $| \mathcal{R} | = r$.

\subsection{Mass-Action Kinetics}

We are particularly interested in the evolution of the concentrations of the chemical species. We will let $x_i = [ \mathcal{A}_i ]$ denote the concentration of the $i^{th}$ species and denote by $\mathbf{x} = [x_1 \; x_2 \; \cdots \; x_m]^T \in \mathbb{R}_{\geq 0}^m$  the \emph{concentration vector}.

The differential equations governing the chemical reactions system (\ref{reaction2}) under the assumption of mass action dynamics can be expressed as
\begin{equation}
\label{de}
\frac{d\mathbf{x}}{dt} = \mathbf{f}(\mathbf{x}) = \sum_{(i,j) \in \mathcal{R}} k(i,j) \: ( \mathbf{z}_j - \mathbf{z}_i ) \: \mathbf{x}^{\mathbf{z}_i}
\end{equation} 
\noindent where $\mathbf{x}^{\mathbf{z}_{i}} = \prod_{j=1}^m x_j^{z_{ij}}$.

Several fundamental properties of chemical kinetics systems are readily seen from this formulation. In particular, it is clear from (\ref{de}) that solutions are not able to wander around freely in $\mathbb{R}^m$. Instead, they are restricted to \emph{stoichiometric compatibility classes}, \cite{H-J1}.

\begin{definition}
\label{stoic}
The \textbf{stoichiometric subspace} for a chemical reaction mechanism (\ref{reaction2}) is the linear subspace $S \subset \mathbb{R}^m$ such that
\[S= \mbox{span} \left\{ \left. (\mathbf{z}_j-\mathbf{z}_i) \; \right| \; (i,j) \in \mathcal{R} \right\}.\]
\noindent The dimension of the stoichiometric subspace will be denoted by $| S | = s$.
\end{definition}

\begin{definition}
The positive \textbf{stoichiometric compatibility class} containing the initial concentration $\mathbf{x}_0 \in \mathbb{R}^m_{>0}$ is the set $\mathsf{C}_{\mathbf{x}_0} = (\mathbf{x}_0 + S) \cap \mathbb{R}^m_{>0}$.
\end{definition}

\begin{proposition} [\cite{H-J1,V-H}]
\label{proposition2}
Let $\mathbf{x}(t)$ be the solution to (\ref{de}) with $\mathbf{x}(0)=\mathbf{x}_0 \in \mathbb{R}^m_{>0}$. Then $\mathbf{x}(t) \in \mathsf{C}_{\mathbf{x}_0}$ for $t \geq 0$.
\end{proposition}

Note that a solution $\mathbf{x}(t)$ of (\ref{de}) with $\mathbf{x}(0)=\mathbf{x}_0 \in \mathbb{R}^m_{>0}$ may exist only on a finite interval $0 \leq t < T$, in which case $\mathbf{x}(t) \in \mathsf{C}_{\mathbf{x}_0}$ for $0 \leq t < T$. Throughout this paper we only consider solutions to (\ref{de}) satisfying $\mathbf{x}(0)=\mathbf{x}_0 \in \mathbb{R}^m_{>0}$, so that Proposition \ref{proposition2} holds.

\subsection{Detailed and Complex Balanced Systems}
\label{detailedsection}

One important characteristic by which we can categorize chemical reaction mechanisms is the nature of the equilibrium concentrations permitted by the mechanism. We start by introducing two such classifications and illustrating how they are related.

\begin{definition}
\label{detailedbalanced}
The concentration $\mathbf{x}^* \in \mathbb{R}^m_{>0}$ is said to be a \textbf{detailed balanced equilibrium concentration} of (\ref{de}) if
\begin{equation}
\label{db}
k(i,j) (\mathbf{x}^*)^{\mathbf{z}_i} = k(j,i) (\mathbf{x}^*)^{\mathbf{z}_j}
\end{equation}
for all $i,j=1,\ldots,n$. A mass-action system is said to be \textbf{detailed balanced} for a given set of rate constants $k(i,j)$ if every positive equilibrium concentration of (\ref{de}) is detailed balanced.
\end{definition} 

\begin{definition}
\label{complexbalanced}
The concentration $\mathbf{x}^* \in \mathbb{R}^m_{>0}$ is said to be a \textbf{complex balanced equilibrium concentration} of (\ref{de}) if
\begin{equation}
\label{cb2}
\sum_{j=1}^{n} k(j,i) (\mathbf{x}^*)^{\mathbf{z}_j} = (\mathbf{x}^*)^{\mathbf{z}_i} \sum_{j=1}^{n} k(i,j)
\end{equation}
for all $i = 1,\ldots,n$. A mass-action system is said to be \textbf{complex balanced} for a given set of rate constants $k(i,j)$ if every positive equilibrium concentration of (\ref{de}) is complex balanced.
\end{definition}

Analysis of complex balanced systems is made easier by the following lemma. An analogous result exists for detailed balanced systems as a consequence of detailed balanced equilibria being a subset of complex balanced equilibria (see Theorem 3.10, \cite{S-M}).

\begin{lemma}[Lemma 5B, \cite{H-J1}]
\label{lemma2}
If a mass action system is complex balanced at some concentration $\mathbf{x}^* \in \mathbb{R}^m_{>0}$, then it is complex balanced at all equilibrium concentrations.
\end{lemma}

It is clear that every detailed or complex balanced equilibrium concentration is an equilibrium concentration of (\ref{de}) and that every detailed balanced equilibrium concentration is also complex balanced. It should be noted, however, that not every equilibrium concentration is a detailed or complex balanced equilibrium concentration, as can be seen by the system
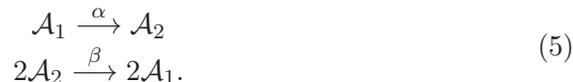
\begin{equation}
\label{system1}
\begin{array}{c}
\mathcal{A}_1 \stackrel{\alpha}{\longrightarrow} \mathcal{A}_2\\
2 \mathcal{A}_2 \stackrel{\beta}{\longrightarrow} 2 \mathcal{A}_1.
\end{array}
\end{equation}
No equilibrium permitted by mechanism (\ref{system1}) is either detailed or complex balanced. Similarly, not every complex balanced equilibrium is a detailed balanced equilibrium, as can be seen by
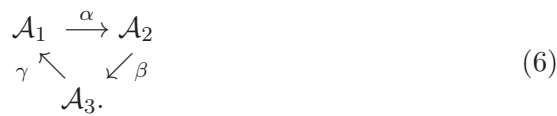
\begin{equation}
\label{system2}
\begin{array}{ccc} \mathcal{A}_1 \; \stackrel{\alpha}{\longrightarrow} \mathcal{A}_2 \\ {}_\gamma \nwarrow \hspace{0.2in} \swarrow_\beta \\ \mathcal{A}_3. \end{array}
\end{equation}
This mechanism permits complex balanced equilibria, but no detailed balanced equilibria.

The structure of the reaction graph is intricately connected to the conditions of detailed and complex balancing of equilibrium points; however, for the sake of brevity we omit such discussion here (for further details, see \cite{F1,H,H-J1,V-H}).

\subsection{Known Stability Results}

In this section, we will discuss some of the known stability results for complex balanced systems. In particular, we state what has come to be known as the \emph{Global Attractor Conjecture} and give several circumstances discussed in the literature under which it is known to hold. 

We start, however, with the main result of \cite{H-J1}.

\begin{theorem}
\label{stabilitytheorem}
If a mass-action system is complex balanced, then there exists within each positive compatibility class $\mathsf{C}_{\mathbf{x}_0}$ a unique positive equilibrium point $\mathbf{x}^*$ which is asymptotically stable.
\end{theorem}
\noindent This is a very powerful result in that it gives sufficient conditions for an extremely predictible---and often desirable in practice---form of behaviour based solely on the nature of the equilibrium points. (Further relationships between the complex balancing condition and the reaction graph of a mechanism are given in \cite{F1,H}.)

The result, however, is local in nature and insufficient to eliminate the possibility of an $\omega$-limit point lying on the boundary of $\mathbb{R}_{>0}^m$. The hypothesis that the unique positive equilibrium point is in fact a \emph{global attractor} for the invariant set $\mathsf{C}_{\mathbf{x}_0}$ was first proposed by Horn in \cite{H1}. For completeness, we state the conjecture here as stated in \cite{C-D-S-S}.

\begin{proposition}[Global Attractor Conjecture]
\label{globalattractorconjecture}
For any complex balanced system and any starting point $\mathbf{x}_0 \in \mathbb{R}_{>0}^m$, the associated complex balanced equilibrium point $\mathbf{x}^*$ of $\mathsf{C}_{\mathbf{x}_0}$ is a global attractor of $\mathsf{C}_{\mathbf{x}_0}$.
\end{proposition}

To date no fully general proof of the conjecture exists; however, several restrictions on long-term behaviour of solutions and special cases under which the conjecture holds are known. In \cite{S-M}, the authors prove the following restriction on the $\omega$-limit set.
\begin{theorem}[Theorem 3.2, \cite{S-M}]
\label{wlimitsettheorem}
For any $\mathbf{x}_0 \in \mathbb{R}_{>0}^m$ of a complex balanced mass-action system, the $\omega$-limit set consists either of complex balanced equilibrium points lying on $\partial \mathbb{R}_{>0}^m$ or of a single positive point of complex balanced equilibrium.
\end{theorem}
\noindent An important consequence of this result is that in order to show the global attractor conjecture holds for a complex balanced system it is sufficient to show that $\omega(\mathbf{x}_0) \cap \partial \mathbb{R}_{>0}^m = \emptyset$ (see also Proposition 19 of \cite{C-D-S-S} and the consequent discussion).

Significant work has also been done recently restricting where $\omega$-limit points may lie on $\partial \mathbb{R}_{>0}^m$. In particular, the following concept (called a \emph{semi-locking set} in \cite{A} and a \emph{siphon} in \cite{A3}) has been used to restrict such points. (The set $L_I$ is formally introduced by Definition \ref{face} in Section \ref{section3}.)

\begin{definition}
\label{semilockingset}
The nonempty index set $I \subseteq \mathcal{S}$ is called a \textbf{semi-locking set} if for every reaction where an element from $I$ is in the product complex, an element from $I$ is also in the reactant complex.
\end{definition}

\begin{lemma}[Theorem 2.5, \cite{A}]
\label{lemma461}
Consider the non-empty index set $I \subseteq \mathcal{S}$. If there exists a $\mathbf{x}_0 \in \mathbb{R}_{>0}^m$ such that $\omega(\mathbf{x}_0) \cap L_I \not= \emptyset$, then $I$ is a semi-locking set.
\end{lemma}
\noindent Consequently, to eliminate the possibility of trajectories approaching $\partial \mathbb{R}_{>0}^m$, it is sufficient to look at sets $L_I$ corresponding to semi-locking sets. Since semi-locking sets can be determined from the reaction graph of the mechanism alone, this is a particularly useful result.

Our work in the next section will focus on the approach used and results obtained in \cite{C-D-S-S}. In this paper, the authors use the novel approach of dividing the state space $\mathbb{R}_{>0}^m$ into \emph{strata}, which are naturally arising partitions of $\mathbb{R}_{>0}^m$. The authors show that within these regions, the trajectories of any detailed balanced system are repelled from the boundary by a linear Lyapunov function of the form $H(\mathbf{x}(t))=\langle \alpha, \mathbf{x}(t) \rangle$. They use this to justify the following partial result on global stability.

\begin{theorem}[Theorem 23, \cite{C-D-S-S}]
\label{detailedstability}
Consider a detailed balancing system whose stoichiometric substace $S$ is two-dimensional and assume that the positive compatibility class $\mathsf{C}_{\mathbf{x}_0}$ is bounded. Then the unique positive equilibrium point $\mathbf{x}^*$ of $\mathsf{C}_{\mathbf{x}_0}$ is a global attractor for $\mathsf{C}_{\mathbf{x}_0}$.
\end{theorem}

In the following section, we will generalize the methodology used in obtaining this result to complex balanced systems. We feel these results represent not only a useful tool for select special cases but a substantial theoretical step toward confirming Proposition \ref{globalattractorconjecture}.

\section{Global Stability}
\label{sectionsection}

In this section, we show how the results of \cite{C-D-S-S} can be extended from detailed balanced systems to complex balanced systems. In Section \ref{section3}, we introduce the necessary background material and generalize the notion of stratifying the state space $\mathbb{R}_{>0}^m$. In Section \ref{section4} we derive the analogous result to Corollary 18 of \cite{C-D-S-S} for cyclic complex balanced systems, and in Section \ref{section5} we extend this result to general complex balanced systems. In Section \ref{section6} we give a few concrete results for determining global stability of complex balanced systems. These results are applied in Section \ref{section7} to a few specific examples.

\subsection{Permutations, Faces and Strata}
\label{section3}

In this section, we introduce the concept of stratifying the positive stoichiometric compatibility classes $\mathsf{C}_{\mathbf{x}_0}$ as it has been used in the literature so far. We then generalize the concept in a natural way so that it can be applied to complex balanced chemical reaction systems.

The idea of stratifying $\mathsf{C}_{\mathbf{x}_0}$ was first introduced in \cite{C-D-S-S} with applications to detailed balanced systems. For such systems, the authors used the sets
\begin{equation}
\label{cdssstrata}
\mathcal{S} = \left\{ \mathbf{x} \in \mathsf{C}_{\mathbf{x}_0} \; \left| \; \left( \frac{\mathbf{x}}{\; \mathbf{x}^*} \right)^{\mathbf{z}_i} > \left( \frac{\mathbf{x}}{\; \mathbf{x}^*} \right)^{\mathbf{z}_j} \right. \mbox{ for } (i,j) \in E' \right\}
\end{equation}
where $\mathbf{x}^*$ is the unique positive equilibrium concentration in $\mathsf{C}_{\mathbf{x}_0}$. The sets $E' \subset \mathcal{R}$ were chosen to contain exactly one of the index pairs $(i,j)$ or $(j,i)$ out of each detailed balanced pair given in Definition \ref{detailedbalanced}. The graph of $E'$ was also required to be acyclic.

While our notion of stratification is based on that presented in \cite{C-D-S-S}, some differences arise. We consider a \emph{complete} ordering of all the complexes in the system, rather than pairwise ordering as in (\ref{cdssstrata}), and we do not require any conditions on the reaction graph. We also keep the notion of stratification general by considering the state space $\mathbb{R}_{>0}^m$ rather than each $\mathsf{C}_{\mathbf{x}_0}$.

First of all, we will need to introduce the concept of a permutation operator.

\begin{definition}
\label{permutation}
Consider the set $I = \left\{1, 2, \ldots, n \right\}$. The operator $\mu: I \mapsto I$ is called a \textbf{permutation operator} if it is bijective. Furthermore, we will say that the permutation operator $\mu$ implies the \textbf{ordering}
\[\mu(i) \; \succ \; \mu(i+1), \hspace{0.2in} i = 1, \ldots, n-1\]
on the set $\left\{ 1, 2, \ldots, n \right\}$.
\end{definition}

A permutation operator simply shuffles the elements of a set. To each such operator we can define a stratum in the following way.

\begin{definition}
\label{strata}
Given a permutation operator $\mu: I \mapsto I$ we define the \textbf{stratum} associated with $\mu$ to be
\begin{equation}
\label{1}
\mathcal{S}_\mu = \left\{ \mathbf{x} \in \mathbb{R}_{>0}^m \; \left| \; \left( \frac{\mathbf{x}}{\; \mathbf{x}^*}\right)^{\mathbf{z}_{\mu(i)}} > \left( \frac{\mathbf{x}}{\; \mathbf{x}^*}\right)^{\mathbf{z}_{\mu(i+1)}} \mbox{for } i = 1, \ldots, n-1 \right. \right\}
\end{equation}
where $\mathbf{x}^*$ is an arbitrary positive equilibrium concentration permitted by the system.
\end{definition}

This is a more general notion of strata than that given by (\ref{cdssstrata}) (i.e. for some systems, strata according to (\ref{cdssstrata}) are further stratified by (\ref{1})); however, it is natural for the analysis we undertake in the remainder of this paper.


Strata defined in this way have some nice properties, most importantly, that each $\mathbf{x} \in \mathbb{R}^m_{>0}$ either belongs to a unique stratum $\mathcal{S}_\mu$ or the boundary separating one or more strata.


It is also worth noting that not every permutation generates a non-empty stratum. For example, for a system containing the complexes $\mathcal{C}_1 = \mathcal{O}, \mathcal{C}_2 = \mathcal{A}_1$, $\mathcal{C}_3 = \mathcal{A}_2,$ and $\mathcal{C}_4 = \mathcal{A}_1 + \mathcal{A}_2$ there are no points satisfying
\[\frac{x_2}{x_2^*} > \frac{x_1}{x_1^*} \frac{x_2}{x_2^*} > \frac{x_1}{x_1^*} > 1\]
since the first and last conditions imply $x_1^*>x_1$ and $x_1>x_1^*$, respectively. That is to say, for the permutation $\mu([1,2,3,4])=[3,4,2,1]$ we have $\mathcal{S}_\mu = \emptyset$ ($\mu([1,2,3,4])=[3,4,2,1]$ will be our short-hand for $\mu(1)=3$, $\mu(2)=4$, $\mu(3)=2$, $\mu(4)=1$). In this paper, we will consider only those permutation operators $\mu$ which generate non-empty strata $\mathcal{S}_\mu$. (This is related to, although not equivalent to, the condition that $E'$ contain no cycles in (\ref{cdssstrata}).)




To answer the question of global stability, we are interested in the behaviour of trajectories near the boundary of the state space $\mathbb{R}_{>0}^m$. Such discussion is aided by partitioning $\partial \mathbb{R}_{>0}^m$ into the following sets $L_I$. (These sets are defined similarly in \cite{A}, \cite{A3}, and \cite{C-D-S-S}. In \cite{A-S}, $L_I$ is denoted $Z_I$. In the standard theory of convex polytopes, $L_I$ is referred to as the relative interior of a face.)

\begin{definition}
\label{face}
Given an index set $I \subseteq \left\{ 1, 2, \ldots, m \right\}$, we will define the set $L_I$ to be
\[L_I = \left\{ \mathbf{x} \in \mathbb{R}_{\geq 0}^m \; | \; x_i = 0 \mbox{ for } i \in I, \mbox{ and } x_i > 0 \mbox{ for } i \not\in I \right\}.\]
\end{definition}


\noindent It should be noted that according to this definition each $\mathbf{x} \in \partial \mathbb{R}_{>0}^m$ can be placed into exactly one $L_I$ so that the $L_I$ uniquely and completely decompose $\partial \mathbb{R}_{>0}^m$.



The following result relates strata and the sets $L_I$. It is based on Lemma 17 of \cite{C-D-S-S}.


\begin{lemma}
\label{alpha}
If $\overline{\mathcal{S}}_\mu \cap L_I \not= \emptyset$ then there exists an $\alpha \in \mathbb{R}^m$ satisfying
\begin{equation}
\label{33}
\begin{array}{l} \alpha_i < 0, \mbox{ for } i \in I \\ \alpha_i = 0, \mbox{ for } i \not\in I \end{array}
\end{equation}
and
\begin{equation}
\label{34}
\langle \mathbf{z}_{\mu(i)} - \mathbf{z}_{\mu(i+1)}, \alpha \rangle \geq 0, \hspace{0.2in} \mbox{for } i = 1, \ldots, n-1.
\end{equation}
\end{lemma}

\begin{proof}

Suppose there is no $\alpha \in \mathbb{R}^m$ satisfying (\ref{33}) and (\ref{34}). By application of Farkas' Lemma on the index set $I$, this implies that there exist $\lambda_i \geq 0$, $i = 1, \ldots, n-1$ such that
\begin{equation}
\label{3890}
\mathbf{v} = \sum_{i=1}^{n-1} \lambda_i(\mathbf{z}_{\mu(i)} - \mathbf{z}_{\mu(i+1)})
\end{equation}
satisfies
\begin{equation}
\label{3891}
\begin{array}{ll} v_i \geq 0, \hspace{0.2in} & \mbox{for all } i \in I \\ v_{i_0} > 0, & \mbox{for at least one } i_0 \in I. \end{array}
\end{equation}


By assumption we have $\overline{\mathcal{S}}_\mu \cap L_I \not= \emptyset$. This implies that there exists a sequence $\left\{ \mathbf{x}^k \right\} \subset \mathcal{S}_\mu$ such that $\mathbf{x}^k \to \mathbf{x} \in L_I$ as $k \to \infty$. By consideration of the quantity $( \mathbf{x} / \mathbf{x}^* )^\mathbf{v}$ separately for $\mathbf{x} \in L_I$ and the sequence $\left\{ \mathbf{x}^k \right\}$ we will produce a contradiction.

Consider $\mathbf{x} \in L_I$. This implies $x_i = 0$ for $i \in I$. Since $v_i \geq 0$ for $i \in I$ and there exists at least one $i_0 \in I$ such that $v_{i_0} > 0$, it follows that
\begin{equation}
\label{9295}
\left( \frac{\mathbf{x}}{\; \mathbf{x}^*} \right)^{\mathbf{v}} = 0.
\end{equation}

Now consider the sequence $\left\{ \mathbf{x}^k \right\} \subset \mathcal{S}_\mu$ converging to $\mathbf{x}$. We have
\[\left( \frac{\mathbf{x}^k}{\; \mathbf{x}^*} \right)^{\mathbf{v}} = \left( \frac{\mathbf{x}^k}{\; \mathbf{x}^*} \right)^{\sum_{i=1}^{n-1} \lambda_i(\mathbf{z}_{\mu(i)} - \mathbf{z}_{\mu(i+1)})} = \prod_{i=1}^{n-1} \left[ \left( \frac{\mathbf{x}^k}{\; \mathbf{x}^*} \right)^{\mathbf{z}_{\mu(i)}-\mathbf{z}_{\mu(i+1)}}\right]^{\lambda_i}.\]
It follows from $\mathbf{x}^k \in \mathcal{S}_\mu$ and $\lambda_i \geq 0$ that, for $i = 1, \ldots, n-1$,
\[\left[ \left( \frac{\mathbf{x}^k}{\; \mathbf{x}^*} \right)^{\mathbf{z}_{\mu(i)}-\mathbf{z}_{\mu(i+1)}} \right]^{\lambda_i} > 1, \hspace{0.2in} \mbox{which implies} \hspace{0.2in} \left( \frac{\mathbf{x}^k}{\; \mathbf{x}^*} \right)^{\mathbf{v}} > 1.\]
It remains to take the limit $\mathbf{x}^k \to \mathbf{x}$. The function $( \mathbf{x} / \mathbf{x}^* )^{\mathbf{v}}$ is continuous on $\mathbb{R}^m_{>0}$; furthermore, it is continuous at any $\mathbf{x} \in \partial \mathbb{R}^m_{>0}$ such that $v_i \geq 0$ if $x_i = 0$. Since $\mathbf{v}$ satisfies this for $\mathbf{x}^k \to \mathbf{x} \in L_I$, we have
\begin{equation}
\label{9296}
\lim_{k \to \infty} \left( \frac{\mathbf{x}^k}{\; \mathbf{x}^*} \right)^\mathbf{v} = \left( \frac{\mathbf{x}}{\; \mathbf{x}^*} \right)^{\mathbf{v}} \geq 1.
\end{equation}

This contradicts (\ref{9295}). It follows that no $\mathbf{v}$ satisfying (\ref{3890}) and (\ref{3891}) exists. However, the existence of such a $\mathbf{v}$ was a direct consequence of the non-existence of an $\alpha$ satisfying (\ref{33}) and (\ref{34}), so it follows that such an $\alpha$ must exist. This proves our claim.

\end{proof}

\subsection{Cyclic Complex Balanced Systems}
\label{section4}

In this section, we consider the properties of cyclic complex balanced systems.

We start by introducing the concept of a reaction cycle as it is used in \cite{H-J1}.

\begin{definition}
\label{reactioncycle}
A family of complex indices $\left\{ \nu_0, \nu_1, \ldots, \nu_l \right\}$, $l \geq 2$, will be called a \textbf{cycle} if
\begin{equation}
\label{22}
\nu_0 = \nu_l
\end{equation}
\noindent but all other members of the family are distinct, and if
\[k(\nu_{j-1},\nu_{j}) > 0, \hspace{0.2in} j=1, 2, \ldots, l\]
\noindent where $l$ is the length of the cycle. The \textbf{reaction cycle} associated with $\left\{ \nu_0, \nu_1, \ldots, \nu_l \right\}$ is defined to be the corresponding set of elementary reactions
\begin{equation}
\label{23}
\mathcal{C}_{\nu_{j-1}} \; \longrightarrow \; \mathcal{C}_{\nu_j}, \hspace{0.2in} j=1, 2, \ldots, l.
\end{equation}
\end{definition}

\begin{definition}
\label{cyclicsystem}
We will say that a mass-action system is \textbf{cyclic} if the system consists only of a single reaction cycle. For such a system, it will be understood that $l=n$.
\end{definition}

In this section, we will consider only cyclic systems. For notational simplicity, we reindex our single cycle of consideration to $\left\{ 1, 2, \ldots, l, 1 \right\}$ so that the reaction cycle in consideration is
\[\mathcal{C}_{j-1} \; \longrightarrow \; \mathcal{C}_j, \hspace{0.2in} j=1, 2, \ldots, l\]
where $\mathcal{C}_{0} = \mathcal{C}_l$.

Reaction cycles are a central topic of consideration in \cite{H-J1}. The following result corresponds to equation (5-10) of that paper. (This should also be contrasted with Lemma 16 in \cite{C-D-S-S} for detailed balanced systems.)

\begin{lemma}
\label{lemmaq}
Consider a cyclic mass-action system. If the system is complex balanced then (\ref{de}) can be written
\begin{equation}
\label{equation}
\frac{d\mathbf{x}}{dt} = \kappa \sum_{i=1}^n \left( \mathbf{z}_{i+1} - \mathbf{z}_i \right) \left( \frac{\mathbf{x}}{\; \mathbf{x}^*} \right)^{\mathbf{z}_i}
\end{equation}
where $\mathbf{x}^*$ is the unique positive equilibrium point guaranteed by complex balancing and $\kappa>0$.
\end{lemma}

\begin{proof}
Since the system is cyclic with the cycle $\left\{ 1, 2, \ldots, n, 1 \right\}$, we can write (\ref{de}) as
\begin{equation}
\label{big2}
\frac{d\mathbf{x}}{dt} = \sum_{i=1}^n k(i,i+1) \: (\mathbf{z}_{i+1} - \mathbf{z}_i) \: \mathbf{x}^{\mathbf{z}_i}
\end{equation}
where $i=n+1$ implies $i=1$. By the assumption of complex balancing, from (\ref{cb2}) we have $k(i-1,i) (\mathbf{x}^*)^{\mathbf{z}_{i-1}} = k(i,i+1) (\mathbf{x}^*)^{\mathbf{z}_i}$ for $i=1, \ldots, n$. This can only be true for a cyclic system if
\begin{equation}
\label{big}
k(1,2) (\mathbf{x}^*)^{\mathbf{z}_1} = k(2,3) (\mathbf{x}^*)^{\mathbf{z}_2} = \cdots = k(n,1) (\mathbf{x}^*)^{\mathbf{z}^n} = \kappa > 0.
\end{equation}
Solving for each $k(i,i+1)$ individually, we have
\[k(i,i+1) = \frac{\kappa}{(\mathbf{x}^*)^{\mathbf{z}_i}}, \hspace{0.2in} \mbox{for } i = 1, \ldots, n\]
which upon substitution into (\ref{big2}) yields (\ref{equation}) and we are done.
\end{proof}

The following result allows us to rearrange the governing system of differential equations given by (\ref{equation}) into a form which will be convenient in light of our conception of strata. Since we are dealing with strata, we will need to recall the definition of a permutation operator (Definition \ref{permutation}).

It is important to notice the difference between $\mu(j+1)$ and $\mu(j)+1$: the increment $\mu(j+1)$ is made with respect to the implied ordering given by the permutation $\mu$, while the increment $\mu(j)+1$ is made with respect to the original ordering of the cycle.

\begin{theorem}
\label{result1}
Given a cyclic complex balanced system and an arbitrary permutation operator $\mu$, the system (\ref{de}) can be written
\begin{equation}
\label{9843}
\frac{d\mathbf{x}}{dt} = \kappa \sum_{i=1}^{n-1} \left[ \sum_{j=1}^i \mathbf{s}_{\mu(j)} \right] \left( \left( \frac{\mathbf{x}}{\; \mathbf{x}^*} \right)^{\mathbf{z}_{\mu(j)}} - \left( \frac{\mathbf{x}}{\; \mathbf{x}^*} \right)^{\mathbf{z}_{\mu(j+1)}} \right)
\end{equation}
where $\mathbf{s}_{\mu(j)} = \mathbf{z}_{\mu(j)+1} - \mathbf{z}_{\mu(j)}$.
\end{theorem}

\begin{proof}
We notice first of all that, since the system is cyclic, we have
\[\sum_{i=1}^n \mathbf{s}_{\mu(i)} = \sum_{i=1}^n \mathbf{s}_{i} = \sum_{i=1}^n (\mathbf{z}_{i+1}-\mathbf{z}_i) = \mathbf{0}\]
which immediately implies $\displaystyle{\kappa \sum_{i=1}^n \mathbf{s}_{\mu(i)} \left( \frac{\mathbf{x}}{\; \mathbf{x}^*} \right)^{\mathbf{z}_{\mu(n)}} = \mathbf{0}.}$

Subtracting this from (\ref{equation}), which is the form of (\ref{de}) justified by Lemma \ref{lemmaq}, we have
\[\begin{split} \frac{d\mathbf{x}}{dt} & = \kappa \sum_{i=1}^{n-1} \mathbf{s}_{\mu(i)} \left[ \left( \frac{\mathbf{x}}{\; \mathbf{x}^*} \right)^{\mathbf{z}_{\mu(i)}} - \left( \frac{\mathbf{x}}{\; \mathbf{x}^*} \right)^{\mathbf{z}_{\mu(n)}} \right] \\ & = \kappa \sum_{i=1}^{n-1} \mathbf{s}_{\mu(i)} \left[ \sum_{j=i}^{n-1} \left( \left( \frac{\mathbf{x}}{\; \mathbf{x}^*} \right)^{\mathbf{z}_{\mu(j)}} - \left( \frac{\mathbf{x}}{\; \mathbf{x}^*} \right)^{\mathbf{z}_{\mu(j+1)}}\right) \right] \\ & = \kappa \sum_{i=1}^{n-1} \left[ \sum_{j=1}^i \mathbf{s}_{\mu(j)} \right] \left( \left( \frac{\mathbf{x}}{\; \mathbf{x}^*} \right)^{\mathbf{z}_{\mu(i)}} - \left( \frac{\mathbf{x}}{\; \mathbf{x}^*} \right)^{\mathbf{z}_{\mu(i+1)}}\right) \end{split}\]
and the result is shown.
\end{proof}

It is clear from (\ref{9843}) that the vectors $\sum_{j=1}^i \mathbf{s}_{\mu(j)}$, $i = 1, \ldots, n-1,$ play an intricate role in determining the dynamics of a system within a given stratum. The following result allows us to further understand the nature of these vectors.

\begin{lemma}
\label{slemma}
For every permutation operator $\mu$ and every $k = 1, 2, \ldots, n-1$, there exist $\lambda_j \in \mathbb{Z}_{\leq 0}$, $j = 1, 2, \ldots, n-1,$ such that
\[\sum_{j=1}^k \mathbf{s}_{\mu(j)} = \sum_{j=1}^{n-1} \lambda_j \left( \mathbf{z}_{\mu(j)} - \mathbf{z}_{\mu(j+1)} \right) \]
where $\mathbf{s}_{\mu(j)} = \mathbf{z}_{\mu(j)+1} - \mathbf{z}_{\mu(j)}$.
\end{lemma}

\begin{proof}

Consider a permutation operator $\mu$ and fix a $k \in \left\{ 1, 2, \ldots, n-1 \right\}$. Consider
\[\mathbf{s}_{\mu(k)} = \mathbf{z}_{\mu(k)+1} - \mathbf{z}_{\mu(k)}.\]
Clearly, there exists a $t_1 \in \left\{ 1, 2, \ldots, n \right\}$ such that $\mu(k)+1=\mu(t_1)$. We need to consider where $\mu(t_1)$ lies in the ordering implied by $\mu$ relative to $\mu(k)$, in particular, whether (1) $\mu(t_1) \succ \mu(k)$, or (2) $\mu(t_1) \prec \mu(k)$. We will use an iterative process on the vectors $\mathbf{s}_{\mu(j)}$, $j=1, \ldots, k,$ to show that the case $\mu(t_1) \succ \mu(k)$ eventually leads us in a natural way to consideration of an index $t_{i_0}$ satisfying $\mu(t_{i_0}) \prec \mu(k)$.

\textbf{Case 1:} If $\mu(t_1) \succ \mu(k)$ then $\mathbf{s}_{\mu(t_1)}$ is a term in the sum $\sum_{j=1}^k \mathbf{s}_{\mu(j)}$. It follows that
\begin{equation}
\label{39}
\begin{split}\mathbf{s}_{\mu(k)}+\mathbf{s}_{\mu(t_1)} & = (\mathbf{z}_{\mu(k)+1}-\mathbf{z}_{\mu(k)})+(\mathbf{z}_{\mu(t_1)+1}-\mathbf{z}_{\mu(t_1)}) \\ & = \mathbf{z}_{\mu(t_1)+1}-\mathbf{z}_{\mu(k)}\end{split}
\end{equation}
since $\mu(k) + 1 = \mu(t_1)$. We now repeat this process. We know that there exists a $t_2 \in \left\{ 1, 2, \ldots, n \right\}$ such that $\mu(t_1)+1 = \mu(t_2)$ and, as before, either $\mu(t_2) \succ \mu(k)$ or $\mu(t_2) \prec \mu(k)$. If $\mu(t_2) \succ \mu(k)$, we add $\mathbf{s}_{\mu(t_2)}$ to the cumulative sum (\ref{39}). We can continue doing this until we arrive at an index $i_0$ for which $\mu(t_{i_0-1})+1=\mu(t_{i_0}) \prec \mu(k)$, yielding
\begin{equation}
\label{crap}
\mathbf{s}_{\mu(k)} + \sum_{i=1}^{i_0-1} \mathbf{s}_{\mu(t_i)} = \mathbf{z}_{\mu(t_{i_0})} - \mathbf{z}_{\mu(k)}.
\end{equation}
We know such a terminal index exists because the cyclic nature of the system guarantees each index $\mu(t_i-1)+1=\mu(t_i) \succ \mu(k)$ is unique, so that a distinct vector $\mathbf{s}_{\mu(t_i)}$ is chosen during each iteration. Since $k < n$ and the cycle is of length $n$, this process must reach an index $\mu(t_{i_0}-1)+1 = \mu(t_{i_0}) \prec \mu(k)$ eventually.

\textbf{Case 2:} If $\mu(t_{i_0}) \prec \mu(k)$, we can interpolate (\ref{crap}) as follows:
\begin{equation}
\label{555}
\begin{split} & \mathbf{s}_{\mu(k)} + \sum_{i=1}^{i_0-1} \mathbf{s}_{\mu(t_i)} = \mathbf{z}_{\mu(t_{i_0})} - \mathbf{z}_{\mu(k)}\\ & \hspace{0.4in} = (\mathbf{z}_{\mu(t_{i_0})} - \mathbf{z}_{\mu(t_{i_0}-1)}) + \cdots + ( \mathbf{z}_{\mu(k+1)} - \mathbf{z}_{\mu(k)}) \\ & \hspace{0.4in} = - (\mathbf{z}_{\mu(k)}-\mathbf{z}_{\mu(k+1)}) - \cdots - (\mathbf{z}_{\mu(t_{i_0}-1)} - \mathbf{z}_{\mu(t_{i_0})}).\end{split}
\end{equation}
Notice that if our initial reindexing $\mu(k)+1 = \mu(t_1)$ yielded $\mu(t_1) \prec  \mu(k)$, we can take $t_{i_0}=t_1$ in the above argument. This amounts to interpolating $\mathbf{s}_{\mu(k)} = \mathbf{z}_{\mu(t_1)}-\mathbf{z}_{\mu(k)}$ directly.

We return now to consideration of the entire sum $\sum_{j=1}^k \mathbf{s}_{\mu(j)}$. Since a distinct vector $\mathbf{s}_{\mu(t_i)}$ is chosen in each application of the argument for Case 1, we can divide this sum into those elements $\mathbf{s}_{\mu(j)}$ considered in (\ref{555}) and those not. For those elements not yet considered, the same argument can be applied starting with the lowest remaining index, which will yield another sum of the form (\ref{555}). This will remove some of the remaining vectors $\mathbf{s}_{\mu(j)}$ from the sum. Since there are a finite number of complexes, this process must terminate at some point. Clearly, any sum of vectors of the form given in (\ref{555}) has non-positive integer coefficients for the terms $\mathbf{z}_{\mu(j)} - \mathbf{z}_{\mu(j+1)}$, so that the existence of $\lambda_j \in \mathbb{Z}_{\leq 0}$ is guaranteed. Since $\mu$ and $k \in \left\{ 1, 2, \ldots, n-1 \right\}$ were chosen arbitrarily, the result follows.

\end{proof}

The results to this point are sufficient to prove the following result. This should be contrasted with Corollary 18 of \cite{C-D-S-S}.

\begin{lemma}
\label{cycliclemma}
Consider a cyclic complex balanced system and an arbitrary permutation operator $\mu$. If $\overline{\mathcal{S}}_\mu \cap L_I \not= \emptyset$ then there exists an $\alpha \in \mathbb{R}_{\leq 0}^m$ satisfying
\[\left.  \begin{array}{l} \alpha_i < 0, \mbox{ for } i \in I \\ \alpha_i = 0, \mbox{ for } i \not\in I \end{array} \right.\]
such that $\langle \alpha, \mathbf{f}(\mathbf{x}) \rangle \leq 0$ for every $\mathbf{x} \in \overline{\mathcal{S}}_\mu$.
\end{lemma}

\begin{proof}
Since $\overline{\mathcal{S}}_\mu \cap L_I \not= \emptyset$, we know by Lemma \ref{alpha} that there exists an $\alpha \in \mathbb{R}_{\leq 0}^m$ satisfying
\[\left.  \begin{array}{l} \alpha_i < 0, \mbox{ for } i \in I \\ \alpha_i = 0, \mbox{ for } i \not\in I \end{array} \right.\]
such that
\[\big\langle \mathbf{z}_{\mu(i)} - \mathbf{z}_{\mu(i+1)}, \alpha \big\rangle \geq 0, \hspace{0.2in} \mbox{for } i = 1, \ldots, n-1.\]

According to Theorem \ref{result1} we have
\begin{equation}
\label{awesome}
\begin{split} \langle \alpha, \mathbf{f}(\mathbf{x}) \rangle & = \Bigg\langle \alpha, \kappa \sum_{i=1}^{n-1} \left[ \sum_{j=1}^i \mathbf{s}_{\mu(j)} \right] \left( \left( \frac{\mathbf{x}}{\; \mathbf{x}^*} \right)^{\mathbf{z}_{\mu(j)}} - \left( \frac{\mathbf{x}}{\; \mathbf{x}^*} \right)^{\mathbf{z}_{\mu(j+1)}} \right) \Bigg\rangle \\ & = \kappa \sum_{i=1}^{n-1} \left( \left( \frac{\mathbf{x}}{\; \mathbf{x}^*} \right)^{\mathbf{z}_{\mu(j)}} - \left( \frac{\mathbf{x}}{\; \mathbf{x}^*} \right)^{\mathbf{z}_{\mu(j+1)}} \right) \cdot \Bigg\langle \alpha, \sum_{j=1}^i \mathbf{s}_{\mu(j)} \Bigg\rangle. \end{split}
\end{equation}
For every $\mathbf{x} \in \overline{\mathcal{S}}_\mu$, by Definition \ref{strata} we have
\begin{equation}
\label{stillawesome}
\left( \left( \frac{\mathbf{x}}{\; \mathbf{x}^*} \right)^{\mathbf{z}_{\mu(j)}} - \left( \frac{\mathbf{x}}{\; \mathbf{x}^*} \right)^{\mathbf{z}_{\mu(j+1)}} \right) \geq 0.
\end{equation}

Now consider $\langle \alpha, \sum_{j=1}^i \mathbf{s}_{\mu(j)} \rangle$. We know from Lemma \ref{slemma} that there exist $\lambda_j \in \mathbb{Z}_{\leq 0}$, $j=1, 2, \ldots, n-1$, such that
\[\sum_{j=1}^i \mathbf{s}_{\mu(j)} = \sum_{j=1}^{n-1} \lambda_j (\mathbf{z}_{\mu(j)}-\mathbf{z}_{\mu(j+1)}).\]
We also know by Lemma \ref{alpha} that $\big\langle \mathbf{z}_{\mu(j)} - \mathbf{z}_{\mu(j+1)}, \alpha \big\rangle \geq 0$. Together, these facts imply that for every $i=1, 2, \ldots, n-1,$
\begin{equation}
\label{alsoawesome}
\Bigg\langle \alpha, \sum_{j=1}^i \mathbf{s}_{\mu(j)} \Bigg\rangle = \sum_{j=1}^{n-1} \lambda_j \big\langle \mathbf{z}_{\mu(j)}-\mathbf{z}_{\mu(j+1)}, \alpha \big\rangle \leq 0.
\end{equation}

It follows immediately from (\ref{awesome}), (\ref{stillawesome}), (\ref{alsoawesome}) and the fact that $\kappa>0$ that $\langle \alpha, \mathbf{f}(\mathbf{x}) \rangle \leq 0$ for every $\mathbf{x} \in \overline{\mathcal{S}}_\mu$, and we are done.

\end{proof}

\subsection{General Complex Balanced Systems}
\label{section5}

In this section, we extend Lemma \ref{cycliclemma} to general complex balanced systems. We follow the methodology employed by Horn \emph{et al.} in generalizing from cyclic complex balanced systems to general complex balanced systems \cite{H-J1}.

The following result extends Lemma \ref{lemmaq} to general complex balanced systems.

\begin{lemma}
\label{yo}
Consider a mass-action system which is complex balanced at $\mathbf{x}^* \in \mathbb{R}_{>0}^m$. Then there exists a $\delta \in \mathbb{Z}_{>0}$ and $\kappa_i > 0$, $i = 1, 2, \ldots, \delta,$ such that
\begin{equation}
\label{yoyoyo}
\frac{d\mathbf{x}}{dt} = \kappa_1 \mathbf{X}_1 + \kappa_2 \mathbf{X}_2 + \cdots + \kappa_\delta \mathbf{X}_\delta
\end{equation}
where
\begin{equation}
\label{equation2}
\mathbf{X}_i = \sum_{j=1}^{l_i} \left(\mathbf{z}_{\nu_{j+1}^{(i)}} - \mathbf{z}_{\nu_{j}^{(i)}}\right) \left( \frac{\mathbf{x}}{\; \mathbf{x}^*} \right)^{\mathbf{z}_{\nu_{j}^{(i)}}}
\end{equation}
where the set $\left\{ \nu_{1}^{(i)}, \nu_{2}^{(i)}, \ldots, \nu_{l_i}^{(i)}, \nu_{l_i+1}^{(i)} \right\}$ is a cycle according to Definition \ref{reactioncycle}.
\end{lemma}

\begin{proof}
According to Lemma 6D of \cite{H-J1}, if a system is complex balanced at $\mathbf{x}^* \in \mathbb{R}_{>0}^m$ then it can be decomposed into a finite number of cyclic subsystems which are all complex balanced at $\mathbf{x}^*$. This decomposition occurs with respect to the rate constants $k(i,j)$, which enter (\ref{de}) linearly. This implies (\ref{de}) can be written as
\[\frac{d\mathbf{x}}{dt} = \mathbf{Y}_1 + \mathbf{Y}_2 + \cdots + \mathbf{Y}_\delta\]
for $\delta \in \mathbb{N}_+$, where each $\mathbf{Y}_i$, $i=1, 2, \ldots, \delta,$ corresponds to a cyclic mass-action system which is complex balanced at $\mathbf{x}^*$. Applying Lemma \ref{lemmaq} to each of these terms yields (\ref{yoyoyo}) with the difference that we must use the ordering of each individual cycle in transforming (\ref{equation}) to (\ref{equation2}).
\end{proof}

We are now prepared to generalize Lemma \ref{cycliclemma} to general complex balanced systems.

\begin{theorem}
\label{maintheorem}
Consider a complex balanced system and an arbitrary permutation operator $\mu$. If $\overline{\mathcal{S}}_\mu \cap L_I \not= \emptyset$ then there exists an $\alpha \in \mathbb{R}_{\leq 0}^m$ satisying
\[\left.  \begin{array}{l} \alpha_i < 0, \mbox{ for } i \in I \\ \alpha_i = 0, \mbox{ for } i \not\in I \end{array} \right.\]
such that $\langle \alpha, \mathbf{f}(\mathbf{x}) \rangle \leq 0$ for every $\mathbf{x} \in \overline{\mathcal{S}}_\mu$.
\end{theorem}

\begin{proof}


Consider a permutation operator $\mu$ satisfying $\overline{\mathcal{S}}_\mu \cap L_I \not= \emptyset$. By Lemma \ref{alpha} there exists an $\alpha \in \mathbb{R}_{\leq 0}^m$ satisfying
\begin{equation}
\label{55}
\left.  \begin{array}{l} \alpha_i < 0, \mbox{ for } i \in I \\ \alpha_i = 0, \mbox{ for } i \not\in I \end{array} \right.
\end{equation}
and
\begin{equation}
\label{56}
\langle \mathbf{z}_{\mu(i)} - \mathbf{z}_{\mu(i+1)}, \alpha \rangle \geq 0, \hspace{0.2in} \mbox{for } i = 1, \ldots, n-1.
\end{equation}
The form of $\alpha$ from (\ref{55}) is what we need for the theorem. We now want to use (\ref{56}) to determine the sign of $\langle \alpha, \mathbf{f}(\mathbf{x}) \rangle$.

Since the system is complex balanced, by Lemma \ref{yo} we have
\[\langle \alpha, \mathbf{f}(\mathbf{x}) \rangle = \kappa_1 \langle \alpha, \mathbf{X}_1 \rangle + \cdots + \kappa_\delta \langle \alpha, \mathbf{X}_\delta \rangle \]
where the $\kappa_i$ are positive constants determined by the rate constants and the $\mathbf{X}_i$ have the form (\ref{equation2}). Each $\mathbf{X}_i$ corresponds to a cycle in the cyclic decomposition of the system where the $i^{th}$ cycle is indexed $\left\{ \nu_1^{(i)}, \nu_2^{(i)}, \ldots, \nu_{l_i}^{(i)}, \nu_1^{(i)} \right\}$. The overall ordering
\begin{equation}
\label{yoyo}
\mu(1) \; \succ \; \mu(2) \; \succ \; \cdots \; \succ \; \mu(n)
\end{equation}
implies an ordering on the complex indices $\left\{ \nu_1^{(i)}, \ldots, \nu_{l_i}^{(i)} \right\}$. We can do this by simply removing the elements from (\ref{yoyo}) which do not correspond to indices in the set $\left\{ \nu_1^{(i)}, \ldots, \nu_{l_i}^{(i)} \right\}$ whilst otherwise preserving the ordering.



Now consider a single term $\langle \alpha, \mathbf{X}_i \rangle$, $i=1, \ldots, \delta$. Firstly, we reindex the complexes so that the relevant cycle is $\left\{ 1, 2, \ldots, l_i, 1 \right\}$. We let $\mu_i$ denote the permutation operator which preserves the ordering implied by $\mu$ on this reduced index set, after reindexing. (For example, consider a system with five complexes and the cycle $\left\{ 2, 4, 1, 2\right\}$. Consider the permutation operator $\mu([1,2,3,4,5]) = [2,5,3,1,4] $. Then we reindex the cycle so that we have $\left\{ 1, 2, 3, 1 \right\}$ and $\mu_i([1,2,3]) = [1, 3, 2]$ since $2 \; \succ \; 1 \; \succ \; 4$ in the original ordering implied by $\mu$.)

Since $\mathbf{X}_i$ is cyclic and complex balanced, we can apply all of the results used in the proof of Lemma \ref{cycliclemma} to get
\begin{equation}
\label{stuff1}
\langle \alpha, \mathbf{X}_i \rangle = \sum_{i=1}^{l_i-1} \left( \left( \frac{\mathbf{x}}{\; \mathbf{x}^*} \right)^{\mathbf{z}_{\mu_i(j)}} - \left( \frac{\mathbf{x}}{\; \mathbf{x}^*} \right)^{\mathbf{z}_{\mu_i(j+1)}} \right) \cdot \Bigg\langle \alpha, \sum_{j=1}^i \mathbf{s}_{\mu_i(j)} \Bigg\rangle
\end{equation}
where $\mathbf{s}_{\mu_i(j)}=\mathbf{z}_{\mu_i(j)+1}-\mathbf{z}_{\mu_i(j)}$. Since the ordering of the complexes corresponding to elements in the $i^{th}$ cycle satisfy (\ref{yoyo}), we have
\[\left( \left( \frac{\mathbf{x}}{\; \mathbf{x}^*} \right)^{\mathbf{z}_{\mu_i(j)}} - \left( \frac{\mathbf{x}}{\; \mathbf{x}^*} \right)^{\mathbf{z}_{\mu_i(j+1)}} \right) \geq 0\]
for all $\mathbf{x} \in \overline{\mathcal{S}}_\mu$. Similarly, we can apply Lemma \ref{slemma} to show that
\begin{equation}
\label{stuff2}
\Bigg\langle \alpha, \sum_{j=1}^i \mathbf{s}_{\mu_i(j)} \Bigg\rangle = \sum_{j=1}^{l_i-1} \lambda_j \big\langle \mathbf{z}_{\mu_i(j)}-\mathbf{z}_{\mu_i(j+1)}, \alpha \big\rangle \leq 0
\end{equation}
where $\lambda_j \in \mathbb{Z}_{\leq 0}$ for $j=1, 2, \ldots, l_i-1$. This implies $\kappa_i \langle \alpha, \mathbf{X}_i \rangle \leq 0$. Since we can carry out this procedure for all $i=1, \ldots, \delta$, we have
\[\langle \alpha, \mathbf{f}(\mathbf{x}) \rangle = \kappa_1 \langle \alpha, \mathbf{X}_1 \rangle + \cdots + \kappa_\delta \langle \alpha, \mathbf{X}_\delta \rangle \leq 0\]
and we are done.



\end{proof}

\subsection{Applications}
\label{section6}

Several global stability results follow immediately from Theorem \ref{maintheorem}. In particular, this theorem is sufficient to guarantee solutions of (\ref{de}) do not approach the boundary for general complex balanced systems if they remain within a single stratum. This is clear because, if we take $T \geq 0$ to be the final time that a trajectory $\mathbf{x}(t)$ enters the relevant stratum, the linear functional $H(\mathbf{x}(t)) = \langle \alpha, \mathbf{x}(t) \rangle$, where $\alpha$ satisfies (\ref{33}), must satisfy $H(\mathbf{x}(t)) \leq H(\mathbf{x}(T)) < 0$ for all $t > T$ since $\frac{d}{dt}H(\mathbf{x}(t)) = \langle \alpha, \mathbf{f}(\mathbf{x}(t)) \rangle \leq 0$ for all $t > T$ by Theorem \ref{maintheorem}. This contradicts the observation that, if $\mathbf{x}(t)$ converges to $\mathbf{x}^* \in L_I$ then
\[\lim_{t \to \infty} H(\mathbf{x}(t)) = H(\mathbf{x}^*) = 0.\]

If multiple strata $\mathcal{S}_\mu$ intersect a given set $L_I$, however, we cannot guarantee the existence of a common $\alpha$ satisfying $\langle \alpha, \mathbf{f}(\mathbf{x}) \rangle \leq 0$ simultaneously within all such strata. Consequently, we cannot rule out the possibility that trajectories approach the boundary through creative maneouvering between strata.

This difficulty, however, does not always arise. The following results will complete the analysis for such systems. We have, however, purposefully kept the first result (Theorem \ref{bigtheorem}) general so it may be applied to systems outside the scope of complex balanced systems and the strata approach taken in this paper. We will make explicit the connection with our systems of interest in a later result (Corollary \ref{sob}).

Throughout this section, when we say that $U$ is a neighbourhood of $K$ in $\mathbb{R}_{\geq 0}^m$ we mean that $U$ is an open covering of $K$ restricted to $\mathbb{R}_{\geq 0}^m$.

\begin{theorem}
\label{bigtheorem}
Consider a general mass-action system with bounded solutions. Suppose that for every set $L_I$ corresponding to a semi-locking set $I$ there exists an $\alpha_I \in \mathbb{R}_{\leq 0}^m$ satisfying
\begin{equation}
\label{condition1}
\begin{array}{l} (\alpha_I)_i < 0, \mbox{ for } i \in I \\ (\alpha_I)_i = 0, \mbox{ for } i \not\in I \end{array}
\end{equation}
and the following property: for every compact subset $K$ of $L_I$, there exists a neighbourhood $U$ of $K$ in $\mathbb{R}_{\geq 0}^m$ such that
\begin{equation}
\label{condition2}
\langle \alpha_I, \mathbf{f}(\mathbf{x}) \rangle \leq 0 \mbox{ for all } \mathbf{x} \in U.
\end{equation}
Then $\omega(\mathbf{x}_0) \cap \partial \mathbb{R}_{>0}^m = \emptyset$ for all $\mathbf{x}_0 \in \mathbb{R}_{>0}^m$.
\end{theorem}

\begin{proof}

We will let $|I|$ denote the number of elements in the set $I$ (i.e. the number of indices $i$ such that $x_i = 0$ for $\mathbf{x} \in L_I$).

We will prove that $\omega(\mathbf{x}_0) \cap \partial \mathbb{R}_{>0}^m = \emptyset$ by showing that $\omega(\mathbf{x}_0) \cap L_I = \emptyset$ for all $I$ from $|I|=m$ to $|I|=1$. This induction corresponds to the dimension of $L_I$ going from $0$ (the origin) to $m-1$. Since $\partial \mathbb{R}_{>0}^m$ is completely partitioned into such sets, this is sufficient to prove the claim.

Our inductive step will consist in showing that $\omega(\mathbf{x}_0) \cap L_{\tilde{I}} \not= \emptyset$ for any semi-locking set $\tilde{I}$ satisfying $|\tilde{I}| = k$, $1 \leq k < m$, implies $(\omega(\mathbf{x}_0) \cap \overline{L}_{\tilde{I}}) \setminus L_{\tilde{I}} \not= \emptyset$. This is sufficient to violate the inductive hypothesis that $\omega(\mathbf{x}_0) \cap L_I = \emptyset$ for all $I$ such that $|I|>k$.


We take $\mathbf{x}_0 \in \mathbb{R}_{>0}^m$ to be arbitrary and fixed throughout the following induction.

\textbf{Base case: } Consider $|I| = m$ (i.e. $I = \mathcal{S}$) and suppose that $I = \mathcal{S}$ is a semi-locking set. We have $L_I=\left\{ \mathbf{0} \right\}$ for which $K=\left\{ \mathbf{0} \right\}$ is trivially a compact subset. By assumption, there exists an $\alpha_I \in \mathbb{R}_{<0}^m$ such that $\langle \alpha_I, \mathbf{f}(\mathbf{x}) \rangle \leq 0$ for all $\mathbf{x} \in U$ where $U$ is some neighbourhood of $K$ in $\mathbb{R}_{\geq 0}^m$. It follows that $\mathbf{x}(t) \in \left\{ \mathbf{x} \in \mathbb{R}_{>0}^m \; | \; \langle \alpha_I, \mathbf{x} \rangle < -\delta \right\}$ for all $t > 0$ and some $\delta > 0$. In other words, we can make a ``cut'' sufficiently close to the origin such that solutions do not enter the cut out area. Consequently $\omega(\mathbf{x}_0) \cap L_I = \emptyset$ for $I = \mathcal{S}$ if $I$ is a semi-locking set.

Since $\omega(\mathbf{x}_0) \cap L_I = \emptyset$ for all $I$ which are not semi-locking sets by Lemma \ref{lemma461}, it follows that $\omega(\mathbf{x}_0) \cap L_I = \emptyset$ for the base case $|I| = m$.

\textbf{Inductive case: } Consider $1 \leq k < m$ and assume that $\omega(\mathbf{x}_0) \cap L_I = \emptyset$ for all $|I| > k$. We will prove that $\omega(\mathbf{x}_0) \cap L_I = \emptyset$ for all $|I| \geq k$.

Assume $\omega(\mathbf{x}_0) \cap L_{\tilde{I}} \not= \emptyset$ for some $\tilde{I}$ such that $|\tilde{I}| = k$ and $\tilde{I}$ is a semi-locking set. Since every $\mathbf{x} \in \overline{L}_{\tilde{I}} \setminus L_{\tilde{I}}$ satisfies $\mathbf{x} \in L_I$ for some $I$ such that $|I| > k$, the inductive hypothesis implies $(\omega(\mathbf{x}_0) \cap \overline{L}_{\tilde{I}}) \setminus L_{\tilde{I}} = \emptyset$, which is equivalent to $(\omega(\mathbf{x}_0) \cap \overline{L}_{\tilde{I}}) \subset L_{\tilde{I}}$. In order to prove the inductive step, we will show that assuming $\omega(\mathbf{x}_0) \cap L_{\tilde{I}} \not= \emptyset$ violates $(\omega(\mathbf{x}_0) \cap \overline{L}_{\tilde{I}}) \subset L_{\tilde{I}}$.

Consider the set $K = \omega(\mathbf{x}_0) \cap \overline{L}_{\tilde{I}}$. Since trajectories are bounded by assumption, $\omega(\mathbf{x}_0)$ is bounded, and consequently $K$ is a compact set. By the inductive hypothesis, this is a subset of $L_{\tilde{I}}$ so that $\langle \alpha_I, \mathbf{f}(\mathbf{x}) \rangle \leq 0$ for all $\mathbf{x} \in U$ where $U$ is some neighbourhood of $K$ in $\mathbb{R}_{\geq 0}^m$.

Consider the linear functional $H(\mathbf{x}) = \langle \alpha_I, \mathbf{x} \rangle$. By (\ref{condition1}) and (\ref{condition2}), $H(\mathbf{x})$ satisfies:
\begin{enumerate}
\item
$H(\mathbf{x}) = 0$ if and only if $\mathbf{x} \in \overline{L}_{\tilde{I}}$,
\item
$H(\mathbf{x}) < 0$ for $\mathbf{x} \in \mathbb{R}_{>0}^m$, and
\item
$\frac{d}{dt} H(\mathbf{x}(t)) = \langle \alpha_I, \mathbf{f}(\mathbf{x}(t)) \rangle \leq 0$ for all $t > 0$ such that $\mathbf{x}(t) \in U$.
\end{enumerate}
\noindent Now consider an arbitrary $\mathbf{y} \in K$. Since $\mathbf{y} \in \omega(\mathbf{x}_0)$, $U$ is a neighbourhood of $\mathbf{y}$, and $H(\mathbf{x})$ is continuous, we can select a sequence $\left\{ t_k \right\}$ $\displaystyle{(\lim_{k \to \infty} t_k = \infty)}$ such that $\left\{ \mathbf{x}(t_k) \right\} \subseteq U$, $\displaystyle{\lim_{k \to \infty}} \mathbf{x}(t_k) = \mathbf{y}$, and $\displaystyle{\lim_{k \to \infty}} H(\mathbf{x}(t_{k})) = H(\mathbf{y}) = 0$.

By Property 3 of $H(\mathbf{x}(t))$ given above, $H(\mathbf{x}(t))$ may not increase while remaining in $U$ and, consequently, in order to approach $\mathbf{y} \in L_{\tilde{I}}$, $\mathbf{x}(t)$ must enter and exit $U$ an infinite number of times. Since $U$ is relatively open in $\mathbb{R}_{\geq 0}^m$ and $\mathbf{x}(t) \in \mathbb{R}_{>0}^m$ for all $t \geq 0$ by Proposition \ref{proposition2}, we can find a sequence $\left\{ \tilde{t}_k \right\}$ corresponding to the entry points $\left\{ \mathbf{x}(\tilde{t}_k) \right\} \subset (\mathbb{R}_{\geq 0}^m \setminus U)$ (i.e. $\mathbf{x}(t) \in U$ for $t \in (\tilde{t}_k,t_k)$). Because trajectories are bounded and $\mathbb{R}_{\geq 0} \setminus U$ is closed, the sequence $\left\{ \mathbf{x}(\tilde{t}_k) \right\}$ has a convergent subsequence on $\mathbb{R}_{\geq 0} \setminus U$. We will denote this sequence $\left\{ \mathbf{x}(\tilde{t}_{k_i}) \right\}$ and let $\tilde{\mathbf{y}}$ be the point such that $\displaystyle{\lim_{i \to \infty}} \mathbf{x}(\tilde{t}_{k_i}) = \tilde{\mathbf{y}} \in \omega(\mathbf{x}_0)$. Since $H(\mathbf{x}(t))$ may not increase for $t \in (\tilde{t}_k,t_k)$ and is bounded above by zero, we have $0 > H(\mathbf{x}(\tilde{t}_k)) \geq H(\mathbf{x}(t_k))$. Since $\displaystyle{\lim_{k \to \infty}}H(\mathbf{x}(t_k)) = 0$, it follows that $\displaystyle{\lim_{k \to \infty}}H(\mathbf{x}(\tilde{t}_k)) = H(\tilde{\mathbf{y}}) = 0$, so that $\tilde{\mathbf{y}} \in \overline{L}_{\tilde{I}}$ by Property 1 of $H(\mathbf{x})$.

In total we have that $\tilde{\mathbf{y}} \in \omega(\mathbf{x}_0) \cap \overline{L}_{\tilde{I}} \cap (\mathbb{R}_{\geq 0} \setminus U) = K \cap (\mathbb{R}_{\geq 0} \setminus U)$. We recall, however, that $U$ was a neighbourhood of $K$ in $\mathbb{R}_{\geq 0}^m$ so that $K \cap (\mathbb{R}_{\geq 0} \setminus U) = \emptyset$. It follows that our original assumption must have been in error, so that $\omega(\mathbf{x}_0) \cap L_{\tilde{I}} = \emptyset$ for all semi-locking sets $\tilde{I}$ satisfying $|\tilde{I}| = k$.

Since $\omega(\mathbf{x}_0) \cap L_I = \emptyset$ for all $I$ which are not semi-locking sets by Lemma \ref{lemma461}, it follows that $\omega(\mathbf{x}_0) \cap L_I = \emptyset$ if $|I|=k$, and our inductive step is shown.

Since $\partial \mathbb{R}_{>0}^m$ can be completely partitioned into sets $L_I$, $1 \leq | I | \leq m$, it follows that
\[\omega(\mathbf{x}_0) \cap \left[ \bigcup_{1 \leq | I | \leq m} L_I \right] = \omega(\mathbf{x}_0) \cap \partial \mathbb{R}_{>0}^m = \emptyset\]
and, since $\mathbf{x}_0 \in \mathbb{R}_{>0}^m$ was chosen arbitrarily, the result follows.

\end{proof}

It should be noted that, since $\mathbb{R}_{>0}^m$ decomposes completely into compatibility classes, Theorem \ref{bigtheorem} is sufficient to guarantee persistence within \emph{all} compatibility classes permitted by the mechanism. This stands in contrast to several results in the literature which present conditions sufficient to guarantee persistence relative to a specified compatibility class $\mathsf{C}_{\mathbf{x}_0}$ but permit other compatibility classes of the same mechanism to be non-persistent (see \cite{A-S} and \cite{C-D-S-S}). Theorem \ref{bigtheorem} can be easily adapted for specific compatibility classes by considering the sets $F_I = \mathsf{C}_{\mathbf{x}_0} \cap L_I$ throughout the argument rather than the sets $L_I$.


We now relate Theorem \ref{bigtheorem} to the methodology of Section \ref{sectionsection}.

\begin{lemma}
\label{lemma192}
Let $M_I$ denote the set of permutation operators $\mu$ such that $\overline{\mathcal{S}}_\mu \cap L_I \not= \emptyset$ for a fixed $I$. Then, for every compact subset $K$ of $L_I$, there exists a neighbourhood $U$ of $K$ in $\mathbb{R}_{\geq 0}^m$ such that
\[\displaystyle{U \subseteq \bigcup_{\mu \in M_I} \overline{\mathcal{S}}_\mu}.\]
\end{lemma}

\begin{proof}
Suppose there is a compact subset $K$ of $L_I$ such that, for every neighbourhood $U$ of $K$ in $\mathbb{R}_{\geq 0}^m$, $U \subseteq \cup_{\mu \in M_I} \overline{\mathcal{S}}_\mu$ is violated. It follows that there exists a sequence $\left\{ \mathbf{x}^k \right\} \subseteq  \cup_{\mu \not\in M_I} \overline{\mathcal{S}}_\mu$ such that $\mathbf{x}^k$ approaches $L_I$ as $k \to \infty$. Since $K$ is compact, we may select the sequence so that $\mathbf{x}^k \to \mathbf{x}$ for some specific $\mathbf{x} \in L_I$.

Since there are finite strata, we can select a subsequence $\left\{ \mathbf{x}^{k_i} \right\} \subseteq \mathcal{S}_\mu$ for a fixed $\mu \not\in M_I$; however, this implies $\displaystyle{\lim_{i \to \infty}} \mathbf{x}^{k_i} = \mathbf{x} \in \overline{\mathcal{S}}_\mu \cap L_I$. This contradicts $\mu \not\in M_I$. Consequently, our assumption was in error, and $U \subseteq \cup_{\mu \in M_I} \overline{\mathcal{S}}_\mu$ for some neighbourhood $U$ of $K$ in $\mathbb{R}_{\geq 0}^m$. The result follows.
\end{proof}

Given Lemma \ref{lemma192} and Theorem \ref{bigtheorem}, we can see that (\ref{condition2}) corresponds to the existence of a common $\alpha_I$ existing in all strata adjacent to a given set $L_I$, which is the desired condition. In general, however, it is difficult to verify this condition directly. The following result gives testable conditions from which (\ref{condition2}) follows. It also answers the question of global stability.

\begin{corollary}
\label{sob}
Consider a complex balanced system. Let $M_I$ denote the set of permutation operators $\mu$ such that $\overline{\mathcal{S}}_\mu \cap L_I \not= \emptyset$ for a fixed $I$. Suppose that for every fixed $I$, $1 \leq |I| < m$, corresponding to a semi-locking set one of Condition 1 or Condition 2 given below is satisfied. Then the unique positive complex balanced equilibrium $\mathbf{x}^*$ of $\mathsf{C}_{\mathbf{x}_0}$ is a global attractor for $\mathsf{C}_{\mathbf{x}_0}$. \\

\noindent \emph{\textbf{Condition 1:}} We will say Condition 1 is satisfied if there exists an $\alpha_I \in \mathbb{R}_{\leq 0}^m$ satisfying (\ref{condition1}) such that, for all $i = 1, 2, \ldots, n-1$ and all $\mu \in M_I$,
\[\langle \mathbf{z}_{\mu(i)} - \mathbf{z}_{\mu(i+1)}, \alpha_I \rangle \geq 0.\]
\\
\noindent \emph{\textbf{Condition 2:}} Consider the cycles $\left\{ \nu_{1}^{(i)}, \nu_{2}^{(i)}, \ldots, \nu_{l_i}^{(i)}, \nu_{1}^{(i)} \right\}$, $i=1, 2, \ldots, \delta$, in the cyclic decomposition of a complex balanced system according to Lemma \ref{yo}. We will reindex each cycle to $\left\{ 1, 2, \ldots, l_i, 1 \right\}$ and let $\mu_i$, $i = 1, \ldots, \delta,$ denote the appropriately reindexed permutation operator restricted to the complexes in the $i^{th}$ cycle. We will say Condition 2 is satisfied if there exists an $\alpha_I \in \mathbb{R}_{\leq 0}^m$ satisfying (\ref{condition1}) such that, for all $i = 1, 2, \ldots, \delta$ and all $\mu \in M_I$,
\[\Bigg\langle \sum_{j = 1}^k \mathbf{s}_{\mu_i(j)}, \alpha_I \Bigg\rangle \leq 0, \hspace{0.2in} \mbox{ for }k = 1, 2, \ldots, l_i-1,\]
where $\mathbf{s}_{\mu_i(j)} = \mathbf{z}_{\mu_i(j)}-\mathbf{z}_{\mu_i(j)+1}$.\\

\end{corollary}

\begin{proof}

The proof will proceed in the following steps. We will prove firstly that Condition 1 or 2 is sufficient to show $\langle \alpha_I, \mathbf{f}(\mathbf{x}) \rangle \leq 0$ for all $\mathbf{x} \in \cup_{\mu \in M_I} \overline{\mathcal{S}}_\mu$. We then show by Lemma \ref{lemma192} that the such systems satisfy the hypotheses of Theorem \ref{bigtheorem} so that $\omega(\mathbf{x}_0) \cap \partial \mathbb{R}_{>0}^m = \emptyset$. We then show that for complex balanced systems this is enough to prove the global stability of the positive equilibrium concentration in each positive compatibility class.

Consider a complex balanced system. We know that $\omega(\mathbf{x}_0) \cap L_I = \emptyset$ for all sets $L_I$ corresponding to non-semi-locking sets $I$ by Lemma \ref{lemma461}.  We also know that for complex balanced systems we have $\omega(\mathbf{x}_0) \cap \left\{ \mathbf{0} \right\} = \emptyset$ (see Proposition 20 of \cite{C-D-S-S}, for one proof). That is to say, we need only consider sets $L_I$ corresponding to semi-locking sets $I$ such that $1 \leq |I| < m$.

It is clear by the proof of Theorem \ref{maintheorem} that either Condition 1 (by (\ref{stuff2})) or Condition 2 (by (\ref{stuff1})) is sufficient to prove that $\langle \alpha_I, \mathbf{f}(\mathbf{x}) \rangle \leq 0$ for all $\mathbf{x} \in \cup_{\mu \in M_I} \overline{\mathcal{S}}_\mu$. This implies by Lemma \ref{lemma192} that for every compact subset $K$ of $L_I$ there is a neighbourhood $U$ of $K$ in $\mathbb{R}_{\geq 0}^m$ such that $\langle \alpha_I, \mathbf{f}(\mathbf{x}) \rangle \leq 0$ for all $\mathbf{x} \in U$. We know that solutions of (\ref{de}) are bounded for complex balanced systems since, for the function
\begin{equation}
\label{globallyapunovfunction}
L(\mathbf{x}) = \sum_{i=1}^m x_i(\ln(x_i)-\ln(x_i^*)-1)+x_i^*,
\end{equation}
we have $\frac{d}{dt}L(\mathbf{x}(t)) < 0$ for all $t \geq 0$ and $\mathbf{x}_0 \in \mathbb{R}_{>0}^m$ \cite{H-J1}. It follows by Theorem \ref{bigtheorem} that $\omega(\mathbf{x}_0) \cap L_I = \emptyset$ for all such sets $L_I$. Since we have considered all sets $L_I$, it follows that $\omega(\mathbf{x}_0) \cap \partial \mathbb{R}_{>0}^m = \emptyset$.

Since our system is complex balanced, it follows by Theorem \ref{stabilitytheorem} that there is precisely one equilibrium concentration $\mathbf{x}^* \in \mathbb{R}_{>0}^m$ in each positive stoichiometric compatibility class $\mathsf{C}_{\mathbf{x}_0}$. Since there are no $\omega$-limit points on the boundary of the positive orthant, by Theorem \ref{wlimitsettheorem} it follows that the only $\omega$-limit point is the positive equilibrium concentration. It follows that $\mathbf{x}^*$ is a global attractor for $\mathsf{C}_{\mathbf{x}_0}$ and we are done.
\end{proof}

Since Condition 1 implies Condition 2 by Lemma \ref{slemma}, but the converse does not necessarily hold, it is typically preferable to check Condition 2. In the following section, our approach will be to define a set $P$ of vectors $\sum_{j = 1}^k \mathbf{s}_{\mu_i(j)}, i=1, \ldots, \delta, k = 1, \ldots, l_i-1,$ and check Condition 2 relative to this set.


The following result corresponds to Corollary 4.5 of \cite{A-S}. It is a generalization of Theorem 23 of \cite{C-D-S-S} (stated Theorem \ref{detailedstability} here) to complex balanced systems.

\begin{corollary}
\label{mainresult}
Consider a complex balanced mass-action system whose stoichiometric subspace $S$ is two-dimensional. Then the unique positive complex balanced equilibrium $\mathbf{x}^*$ of $\mathsf{C}_{\mathbf{x}_0}$ is a global attractor for $\mathsf{C}_{\mathbf{x}_0}$.
\end{corollary}

\begin{proof}
With application of Corollary \ref{sob}, the proof follows identically to the proof of Theorem 23 contained in \cite{C-D-S-S}. We also notice that since trajectories of any complex balanced system are bounded by $L(\mathbf{x}(t)) \leq L(\mathbf{x}_0)$ for all $t \geq 0$, we may remove the assumption of boundedness. 
\end{proof}

\subsection{Examples}
\label{section7}

In this section, we present two examples. The first is an example where Corollary \ref{sob} can be applied while the second is an example where it cannot. \\

\textbf{Example 1:} The following example is given in \cite{A-S} as an example of a three-dimensional complex balanced system for which a general method of guaranteeing global stability is not known. The system considered is
\begin{equation}
\label{example1}
\mathcal{A}_1 \; \leftrightarrows \; \mathcal{A}_2 \; \leftrightarrows \; \mathcal{A}_1 + \mathcal{A}_2 \; \leftrightarrows \; \mathcal{A}_1 + \mathcal{A}_3.
\end{equation}
We assign $\mathcal{C}_1 = \mathcal{A}_1$, $\mathcal{C}_2 = \mathcal{A}_2$, $\mathcal{C}_3 = \mathcal{A}_1 + \mathcal{A}_2$, and $\mathcal{C}_4 = \mathcal{A}_1 + \mathcal{A}_3$, and $x_1 = [\mathcal{A}_1]$, $x_2 = [\mathcal{A}_2]$, and $x_3 = [\mathcal{A}_3]$. The system is complex balanced at all equilibrium concentrations so we need not consider the rate constants. The compatibility class $\mathsf{C}_{\mathbf{x}_0} = \mathbb{R}_{>0}^3$ is three-dimensional and the only non-trivial semi-locking set is $I = \left\{ 1, 2 \right\}$ so that we need only consider the set $L_I$ corresponding to this index set.

We will show that all strata such that $\overline{\mathcal{S}}_\mu \cap L_{\left\{1, 2 \right\}} \not= \emptyset$ have a common $\alpha_I \in \mathbb{R}_{\leq 0}^m$ satisfying (\ref{condition1}) and Condition 2 of Corollary \ref{sob}. There are six $\mu$ such that $\overline{\mathcal{S}}_\mu \cap L_{\left\{ 1, 2 \right\}} \not= \emptyset$:
\[\begin{array}{ll}
(1) \;\; \mu ([1,2,3,4]) = [2,4,1,3] \hspace{0.5in} & (4) \;\; \mu ([1,2,3,4]) = [2,1,4,3] \\ (2) \;\; \mu ([1,2,3,4]) = [4,2,1,3] \hspace{0.3in} & (5) \;\; \mu ([1,2,3,4]) = [1,2,4,3] \\ (3) \;\; \mu ([1,2,3,4]) = [4,1,2,3] & (6) \;\; \mu ([1,2,3,4]) = [1,4,2,3]. \end{array}\]
Since the vectors $\sum_{j=1}^k \mathbf{s}_{\mu_i(j)}$ are the vector coefficients of the bracketed strata terms in (\ref{9843}), it is instructive to rewrite the system of differential equations (\ref{de}) implied by the network (\ref{example1}) according to Theorem \ref{result1}. (This analysis is not, however, required to determine the set of all admissible vectors $\sum_{j=1}^k \mathbf{s}_{\mu_i(j)}$.) We will carry out the analysis for one stratum and leave the rest as an exercise.

The first stratum is given by
\begin{equation}
\label{05}
\mathcal{S}_{\mu} = \left\{ \mathbf{x} \in \mathbb{R}_{>0}^3 \; \left| \; \frac{x_2}{x_2^*} > \frac{x_1}{x_1^*} \cdot \frac{x_3}{x_3^*} > \frac{x_1}{x_1^*} > \frac{x_1}{x_1^*} \cdot \frac{x_2}{x_2^*} \right. \right\}.
\end{equation}
Since the system can be decomposed into the cycles $\left\{ 1, 2, 1 \right\}$, $\left\{ 2, 3, 2 \right\}$, and $\left\{ 3, 4, 3 \right\}$, according to Lemma \ref{yo} and Theorem \ref{result1}, the system (\ref{de}) can be written
\begin{equation}
\label{06}
\begin{split} & \frac{d\mathbf{x}}{dt} = \kappa_1 \left[ \begin{array}{c} 1 \\ -1 \\ 0 \end{array} \right] \left( \frac{x_2}{x_2^*} - \frac{x_1}{x_1^*} \right) + \kappa_2 \left[ \begin{array}{c} 1 \\ 0 \\ 0 \end{array} \right] \left( \frac{x_2}{x_2^*} - \frac{x_1}{x_1^*} \cdot \frac{x_2}{x_2^*} \right) \\ & \hspace{1.5in} + \kappa_3 \left[ \begin{array}{c} 0 \\ 1 \\ -1 \end{array}\right] \left( \frac{x_1}{x_1^*} \cdot \frac{x_3}{x_3^*} - \frac{x_1}{x_1^*} \cdot \frac{x_2}{x_2^*}\right) \end{split}
\end{equation}
where $\mathbf{x}^* = [x_1^*, x_2^*, x_3^*]^T$ is the unique positive complex balanced equilibrium point and $\kappa_1$, $\kappa_2$ and $\kappa_3$ are positive constants determined by the rate constants. In $\mathcal{S}_{\mu}$ the bracketed terms of (\ref{06}) are strictly positive so that the sign of $\langle \alpha_I, \mathbf{f}(\mathbf{x}) \rangle$ is determined by the vector terms alone. Consider a vector $\alpha_I \in \mathbb{R}_{\leq 0}^3$ satisfying (\ref{condition1}) for which 
\[\alpha_I = \lambda_1 (-1, 0, 0) + \lambda_2 (-1, -1, 0), \; \; \; \; \; \lambda_1 \geq 0, \lambda_2 \geq 0.\]
For any such $\alpha_I$ we have $\langle \alpha_I, \mathbf{f}(\mathbf{x}) \rangle \leq 0$, which is sufficient to show the linear function $H(\mathbf{x}(t)) = \langle \alpha_I, \mathbf{x}(t) \rangle$ repels trajectories from the set $L_{\left\{ 1, 2 \right\}}$ in the first stratum.

A similar analysis can be carried out in the five other strata. Removing repetition, the set of admissible vectors $\sum_{j=1}^k \mathbf{s}_{\mu(j)}$ is
\[ P = \left\{ \left[ \begin{array}{c} 1 \\ -1 \\ 0 \end{array} \right], \left[ \begin{array}{c} -1 \\ 1 \\ 0 \end{array} \right], \left[ \begin{array}{c} 1 \\ 0 \\ 0 \end{array} \right], \left[ \begin{array}{c} 0 \\ 1 \\ -1 \end{array} \right] \right\}.\]
Since $\alpha_I = (-1,-1,0)$ satisfies $\langle \alpha_I, \mathbf{v} \rangle \leq 0$ for all $\mathbf{v} \in P$, we have that $\langle \alpha_I, \mathbf{f}(\mathbf{x}) \rangle \leq 0$ for all $\mathbf{x} \in \overline{\mathcal{S}}_\mu$ where $\mathcal{S}_\mu$ is such that $\overline{\mathcal{S}}_\mu \cap L_{\left\{ 1, 2 \right\}} \not= \emptyset$. It follows by Corollary \ref{sob} that $\mathbf{x}^*$ is a global attractor for $\mathsf{C}_{\mathbf{x}_0} = \mathbb{R}_{>0}^3$ and we are done. \\

\textbf{Example 2:} Consider the system
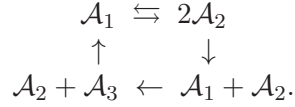
\[\begin{array}{c} \mathcal{A}_1 \; \leftrightarrows \; 2\mathcal{A}_2 \\ \uparrow \hspace{0.5in} \downarrow \\ \mathcal{A}_2 + \mathcal{A}_3 \; \leftarrow \; \mathcal{A}_1 + \mathcal{A}_2. \end{array}\]
We assign $\mathcal{C}_1 = \mathcal{A}_1$, $\mathcal{C}_2 = 2\mathcal{A}_2$, $\mathcal{C}_3 = \mathcal{A}_1 + \mathcal{A}_2$, and $\mathcal{C}_4 = \mathcal{A}_2 + \mathcal{A}_3$, and $x_1 = [\mathcal{A}_1]$, $x_2 = [\mathcal{A}_2]$, and $x_3 = [\mathcal{A}_3]$. The system is complex balanced at all equilibrium concentrations so we need not consider the rate constants. As in the previous example, the compatibility class $\mathsf{C}_{\mathbf{x}_0} = \mathbb{R}_{>0}^3$ is three-dimensional and the only non-trivial semi-locking set is $I = \left\{ 1, 2 \right\}$ so that we need only consider the set $L_I$ corresponding to this index set.

We will show that there is no common $\alpha_I$ satisfying (\ref{condition1}) and Condition 2 of Corollary \ref{sob} for all strata such that $\overline{\mathcal{S}}_\mu \cap L_{\left\{ 1, 2 \right\}} \not= \emptyset$. Notice that since Condition 1 of Corollary \ref{sob} implies Condition 2 by Lemma \ref{slemma}, this is sufficient to show that neither condition is satisfied. There are five $\mu$ such that $\overline{\mathcal{S}}_\mu \cap L_{\left\{ 1, 2 \right\}} \not= \emptyset$:
\[\begin{array}{ll}
(1) \;\; \mu ([1,2,3,4]) = [1,4,2,3] \hspace{0.5in} & (4) \;\; \mu ([1,2,3,4]) = [1,4,3,2] \\ (2) \;\; \mu ([1,2,3,4]) = [4,1,2,3] \hspace{0.3in} & (5) \;\; \mu ([1,2,3,4]) = [4,1,3,2]. \\ (3) \;\; \mu ([1,2,3,4]) = [4,2,1,3] & \end{array}\]

Again, we carry out the analysis for the first stratum and leave the rest as an exercise. The stratum is given by
\begin{equation}
\label{005}
\mathcal{S}_{\mu} = \left\{ \mathbf{x} \in \mathbb{R}_{>0}^3 \; \left| \; \frac{x_1}{x_1^*} > \frac{x_2}{x_2^*} \cdot \frac{x_3}{x_3^*} > \left( \frac{x_2}{x_2^*} \right)^2 > \frac{x_1}{x_1^*} \cdot \frac{x_2}{x_2^*} \right. \right\}
\end{equation}
where $\mathbf{x}^* = [x_1^*, x_2^*, x_3^*]^T$ is the unique positive complex balanced equilibrium point.

Since the system can be decomposed into the cycles $\left\{ 1, 2, 1 \right\}$ and \\ $\left\{ 1, 2, 3, 4, 1 \right\}$, according to Lemma \ref{yo} and Theorem \ref{result1}, the system (\ref{de}) can be written
\begin{equation}
\label{09}
\begin{split} & \frac{d\mathbf{x}}{dt} = \kappa_1 \left[ \begin{array}{c} -1 \\ 2 \\ 0 \end{array} \right] \left( \frac{x_1}{x_1^*} - \left( \frac{x_2}{x_2^*} \right)^2 \right) + \kappa_2 \left\{ \left[ \begin{array}{c} -1 \\ 2 \\ 0 \end{array} \right] \left( \frac{x_1}{x_1^*} - \frac{x_2}{x_2^*} \cdot \frac{x_3}{x_3^*} \right) \right. \\ & \hspace{0.5in} + \left. \left[ \begin{array}{c} 0 \\ 1 \\ -1 \end{array}\right] \left( \frac{x_2}{x_2^*} \cdot \frac{x_3}{x_3^*} - \left( \frac{x_2}{x_2^*} \right)^2 \right) + \left[ \begin{array}{c} 1 \\ 0 \\ -1 \end{array}\right] \left( \left( \frac{x_2}{x_2^*} \right)^2 - \frac{x_1}{x_1^*} \cdot \frac{x_2}{x_2^*} \right) \right\} \end{split}
\end{equation}
where $\kappa_1$ and $\kappa_2$ are positive constants determined by the rate constants. Consider a vector $\alpha_I \in \mathbb{R}_{\leq 0}^3$ satisfying (\ref{condition1}) for which
\[\alpha_I = \lambda_1 (-2, -1, 0) + \lambda_2 (0, -1, 0), \; \; \; \; \; \lambda_1 \geq 0, \lambda_2 \geq 0.\]
For any such $\alpha_I$ we have $\langle \alpha_I, \mathbf{f}(\mathbf{x}) \rangle \leq 0$ for all $\mathbf{x} \in \overline{\mathcal{S}}_{\mu}$.

A similar analysis can be carried out in the four other strata. Removing repetition, the set of admissible vectors $\sum_{j=1}^k \mathbf{s}_{\mu(j)}$ is
\[ P = \left\{ \left[ \begin{array}{c} -1 \\ 2 \\ 0 \end{array} \right], \left[ \begin{array}{c} 0 \\ 1 \\ -1 \end{array} \right], \left[ \begin{array}{c} 1 \\ 0 \\ -1 \end{array} \right], \left[ \begin{array}{c} 1 \\ -1 \\ -1 \end{array} \right], \left[ \begin{array}{c} 1 \\ -2 \\ 0 \end{array} \right], \left[ \begin{array}{c} 2 \\ -2 \\ -1 \end{array} \right], \left[ \begin{array}{c} -1 \\ 1 \\ 0 \end{array} \right] \right\}.\]
Clearly, there is no $\alpha_I$ satisfying (\ref{condition1}) such that $\langle \alpha_I, \mathbf{v} \rangle \leq 0$ for all $\mathbf{v} \in P$ and consequently Corollary \ref{sob} cannot be applied.

\section{Conclusions}

In this paper, we have presented several results which extend the breadth of globally stable complex balanced systems (Theorem \ref{bigtheorem}) and contributed several important theoretical steps toward proving the \emph{Global Attractor Conjecture} (Proposition \ref{globalattractorconjecture}). In particular, we have extended the notion of stratification of the state space $\mathbb{R}_{>0}^m$ to a form applicable to complex balanced systems (Definition \ref{strata}) and shown that, while trajectories lie within a particular stratum $\mathcal{S}_\mu$, they are necessarily repelled from any adjacent set $L_I$, the relative interior of a face, by the linear Lyapunov function $H(\mathbf{x}(t)) = \langle \alpha, \mathbf{x}(t) \rangle$ (Theorem \ref{maintheorem}).

The main result of this paper, Theorem \ref{maintheorem}, is limited in that it prevents trajectories from approaching the boundary from within a single stratum but does not necessarily prevent convergence to the boundary for trajectories which continually jump between strata. This is because the linear Lyapunov function $H(\mathbf{x}(t))=\langle \alpha, \mathbf{x}(t) \rangle$ which repels trajectories from the boundary is specific to each stratum $\mathcal{S}_\mu$. Even in cases where every $H(\mathbf{x}(t)) = \langle \alpha, \mathbf{x}(t) \rangle$ individually bounds trajectories away from a common set $L_I$, it cannot be ruled out without further work that trajectories approach $L_I$ through creative maneouvring between strata. We presented one result (Corollary \ref{sob}) where this difficulty could be resolved and therefore global stability could be shown.

It does not seem probable, however, that a trajectory should be permitted to approach a set $L_I$ even without a common $\alpha$ satisfying (\ref{33}) and (\ref{34}) for all $\mathcal{S}_\mu$ such that $\overline{\mathcal{S}}_\mu \cap L_I \not= \emptyset$. This is especially true in light of the restrictions placed on trajectories by the global Lyapunov function $L(\mathbf{x}(t))$ given by (\ref{globallyapunovfunction}). As such, we feel the solution to the global attractor conjecture lies in a fuller understanding of the geometry of the strata $\mathcal{S}_\mu$, the relationship to the global Lyapunov function $L(\mathbf{x}(t))$, and how the functions $H(\mathbf{x}(t))$ relate to one another for different strata.\\

\textbf{Acknowledgements:} We would like to thank the reviewers for their many helpful suggestions, as well as Anne Shiu for her attentive and thorough reads through the manuscript.

\addtocontents{toc}{\protect\contentsline{chapter}{\vspace{0.15in} \textbf{Bibliography}}{}}


\begin{thebibliography}{99}

\bibitem{A} D. Anderson, \emph{Global Asymptotic Stability for a Class of Nonlinear Chemical Equations,} SIAM J. Appl. Math. 68 (2008), no. 5, pp. 1464--1476.

\bibitem{A-S} D. Anderson and A. Shiu, \emph{The dynamics of weakly reversible population processes near facets,} SIAM J. Appl. Math. 70 (2010), no. 6, pp. 1840-1858.

\bibitem{A3} D. Angeli, P. Leenheer, and E. Sontag, \emph{A Petri net approach to the study of persistence in chemical reaction networks,} Math. Biosci. 210 (2007), no. 2, pp. 598--618.



\bibitem{B-B1} A. Bamberger and E. Billette, \emph{Quelques extensions d'un th\'{e}or\`{e}me de Horn et Jackson,} C. R. Acad. Sci. Paris S\'{e}r. I Math. 319 (1994), no. 12, pp. 1257--1262.



\bibitem{C-D-S-S} G. Craciun, A. Dickenstein, A. Shiu, and B. Sturmfels, \emph{Toric Dynamical Systems,} J. Symbolic Comput. 44 (2009), no. 11, pp. 1551--1565.



\bibitem{F1} M. Feinberg, \emph{Complex balancing in general kinetic systems,} Arch. Rational Mech. Anal. 49 (1974), pp. 187--194.

\bibitem{F2} M. Feinberg, \emph{The existence and uniqueness of steady states for a class of chemical reaction networks,} Arch. Rational Mech. Anal. 132 (1995), no. 4, pp. 311--370.


\bibitem{H} F. Horn, \emph{Necessary and sufficient conditions for complex balancing in chemical kinetics,} Arch. Rational Mech. Anal. 49 (1972), pp. 172--186.

\bibitem{H1} F. Horn, \emph{The dynamics of open reaction systems: Mathematical aspects of chemical and biochemical problems and quantum chemistry,} pp. 125-137, SIAM-AMS Proceedings. Vol. VIII, Amer. Math. Soc., Providence, R.I., 1974.

\bibitem{H-J1} F. Horn and R. Jackson, \emph{General mass action kinetics,} Arch. Rational Mech. Anal. 47 (1972), pp. 81--116.







\bibitem{S} E. Sontag, \emph{Structure and stability of certain chemical networks and applications to the kinetic proofreading model of T-cell receptor signal transduction,} IEEE Trans. Automat. Control 46 (2001), no. 7, pp. 1028--1047.

\bibitem{S-C} D. Siegel and Y.F. Chen, \emph{Global stability of deficiency zero chemical networks,} Canad. Appl. Math. Quart. 2 (1994), no. 3, pp. 413--434.

\bibitem{S-M} D. Siegel and D. MacLean, \emph{Global stability of complex balanced mechanisms,} J. Math. Chem. 27 (2000), no. 1-2, pp. 89--110.



\bibitem{V-H} A.I. Vol'pert and S.I. Hudjaev, \emph{Analysis in Classes of Discontinuous Functions and Equations of Mathematical Physics}, chapter 12. Martinus Nijhoff Publishers, Dordrecht, Netherlands, 1985.


\end{thebibliography}
\end{document}